\documentclass[12pt,reqno]{amsart}

\usepackage{amssymb}
\usepackage{amsmath}
\usepackage{cite}
\usepackage{graphicx}

\makeatletter  
\newif\if@restonecol  
\makeatother

\usepackage[linesnumbered,ruled,vlined]{algorithm2e}
\usepackage{algorithmicx}  
\usepackage{algpseudocode} 

\usepackage{amssymb, amsmath, amsfonts, latexsym}
\usepackage{enumerate}
\usepackage[notref,notcite]{showkeys}

\RequirePackage[OT1]{fontenc}
\RequirePackage{amsthm,amsmath}
\RequirePackage[numbers]{natbib}
\RequirePackage[colorlinks,citecolor=blue,urlcolor=blue]{hyperref}

\def\var{{\mathrm{{\rm Var}}}}
\def\cov{{\mathrm{{\rm Cov}}}}

\newtheorem{thm}{Theorem}[section]

\newtheorem{lem}[thm]{Lemma}
\newtheorem{prop}[thm]{Proposition}
\newtheorem{example}[thm]{Example}

\theoremstyle{definition}
\newtheorem{defn}{Definition}[section]
\newtheorem{rem}{Remark}[section]
\numberwithin{equation}{section}

\numberwithin{equation}{section}
\theoremstyle{plain}

\newcommand{\ee}{\mathbb{E}}
\newcommand{\mm}{\mathbb{M}}
\newcommand{\nn}{\mathbb{N}}
\newcommand{\rr}{\mathbb{R}}
\newcommand{\pp}{\mathbb{P}}

\def\FF{\mathcal F}
\def\EE{\mathcal E}

\def\MM{\mathcal M}
\def\NN{\mathcal N}
\def\SS{\mathcal S}

\title{ Maximum entropy estimator for Hidden Markov models: reduction to dimension 2}

\author{Shulan HU}\address{School of Statistics and Mathematics, Zhongnan University of Economics and Law}\email{hu\_shulan@zuel.edu.cn}

\author{Xinyu WANG}\address{Wenlan School of Business, Zhongnan University of Economics and Law}\email{wang\_xin\_yu@zuel.edu.cn}

\author{Liming Wu}\address{Universit\'e Clermont Auvergne, Laboratoire de math\'ematiques Blaise Pascal}\email{Li-Ming.Wu@math.univ-bpclermont.fr}

\begin{document}

\maketitle

\begin{abstract}
In this paper, we introduce the maximum entropy estimator (MEE in short) $\hat \theta^{ME}_n$ based on 2-dimensional empirical distribution of the observation sequence $(y_0,y_1,\cdots,y_n)$ of a Hidden Markov Model (HMM in short), when the sample size is big: in that case the maximum likelihood estimator (MLE)  is too consuming in time by the classical Baum-Welch EM algorithm. We prove the consistency and the asymptotic normality of $\hat\theta^{ME}_n$ in a quite general framework, where the asymptotic covariance matrix is explicitly estimated in terms of the 2-dimensional Fisher information. To complement it we use also the 2-dimensional relative entropy to study the hypotheses testing problem. Furthermore we propose the  gradient descent of  2-dimensional relative entropy (2RE) algorithm for finding $\hat \theta_n^{ME}$, which works for very big $n$ and big number $m$ of the hidden states. Some numerical examples are furnished and commented for illustrating our theoratical results.
\end{abstract}

\section{Introduction}
\subsection{Background}
A hidden Markov model (HMM) is a discrete-time bivariate stochastic process $\{Z_n=(X_n,Y_n)\}_{n\geq0}$ :
$$
\begin{array}[c]{ccccccccc}
\text{observed signals: \ }& Y_0 & &Y_1& &\cdots & &Y_n &\cdots\\
&\uparrow& &\uparrow& &\cdots & &\uparrow& \cdots\\
\text{unobserved states : \ }& X_0& \rightarrow& X_1&\rightarrow& \cdots& \rightarrow& X_n& \cdots
\end{array}
$$
where

\begin{itemize}
	\item $(X_n)_{n\ge0}$ is a Markov chain valued in $\mm=\{1,\cdots,m\}$ $(m\ge 2)$, with the transition probability matrix $P=(p_{ij})_{i,j\in \mm}$ (i.e. $p_{ij}=\pp(X_{n+1}=j|X_n=i)$);
	
	\item Given $X_{[0,n]}:=(X_0,\cdots, X_n)=(x_0,\cdots, x_n)=:x_{[0,n]}$, $Y_0,\cdots,Y_n$ are conditionally independent, with values in some signal space $S$, and the conditional law of $Y_n$ is given by
	$$
	q_{\beta_{i}}=q_{\beta_{i}}(y) \sigma(dy), \ i=x_n
	$$
	where $\{q_\beta;\beta\in \bar O\}$ is a family of probability measures on $S$ with the parameter $\beta$ varying in the closure $\bar O$ of some open and bounded subset $O$ in $\rr^d$,  $\sigma(dy)$ is some reference positive measure (e.g. the counting measure or the Lebesgue measure according to $S$ is at most countable or the Euclidean space).
\end{itemize}

The transition probabilities $p_{ij}$ and the parameters $\beta_i$ in the signal distributions depend on some unknown parameter $\theta=(\theta_1,\cdots,\theta_M)$, i.e. $p_{ij}=p_{ij}(\theta)$, $\beta_i=\beta_{i}(\theta)$, where $\theta$ varies in the closure $\Theta$ of some bounded and open subset $\Theta^0$ of $\rr^M$. Often $p_{ij}, \beta_i$ are all unknown, in that case $\theta=((p_{ij})_{i\ne j}, (\beta_i))$ has $M=m^2-m+md$ unknown parameters. The objective for the statistical inferences of HMMs is to estimate or to determine $\theta$ through the observation sequence $Y_{[0,n]}=y_{[0,n]}$.

The underlying process $\{X_n\}$ is  often referred to as the {\it regime}, the signals process $\{Y_n\}$ as the {\it observation sequence}. HMMs were  introduced by Baum and Petrie \cite{BP1966}(1966).
HMMs have been widely studied and used in statistics and information theory. HMMs are found of useful applications across a range of science and engineering domains, including genomics, speech recognition, signal processing, optical character recognition, machine translation, computer vision, finance and economics etc. See the surveys by Rabiner \cite{Rabiner89}(1989) and Ephraim \cite{Ephraim02}(2002).

Baum and Petrie \cite{BP1966} and Petrie \cite{Petrie69}
studied statistical inference of finite-state finite-signal HMMs by proving the identifiability of the HMM and the consistency, the asymptotic normality of the maximum likelihood estimator (MLE in short) $\theta_n^{ML}$ of the unknown parameters $\theta$. Baum, Petrie, Soules and Weiss \cite{BPSW1970}(1970) introduced the forward-backward algorithm for calculating the conditional distribution of $X_k$ knowing the observation sequence, and developed the so-called Baum-Welch EM (expectation-maximization) algorithm for finding the MLE. Those two algorithms, together with the Viterbi algorithm (Viterbi \cite{Viterbi67} (1967)) for finding ${\rm argmax}_{x_{[0,n]}}p(x_{[0,n]}|y_{[0,n]})$, the decoding problem in information theory,   constitute the box of the three main computation tools for HMMs.

Generalizations to more general HMMs (with continuous signal or continuous state, or switching HMMs etc) or studies of new problems are realized during the last fifty years:  identifiability of a HMM (\cite{BK1957}, \cite{Teicher1963}, \cite{FR1992}, \cite{R1996-2}, \cite{BB1998}), new ergodic theorems for relative entropy densities of HMMs (\cite{L1992}, \cite{GM2000}, \cite{DM2001}), consistency and asymptotic normality of the MLE (\cite{BBSM1986},\cite{L1992},\cite{R1994},\cite{R1995-1},\cite{BR1996},\cite{BRR1998},\cite{DM2001},\cite{DMOH2011},\cite{AHL2016}), algorithms for estimating the state, parameter, number of states (\cite{DLR1977},\cite{R1995-2},\cite{R1997},\cite{BB1998},\cite{NG2001},\cite{GWR2007},\cite{MA2013}), exponential forgetting of the predictor (\cite{GM2000}), large deviations (\cite{HW2011}) and concentration inequalities (\cite{H2011}) etc.

\subsection{Motivation}
The main difficulty for the statistical studies and applications of HMMs is:  the likelihood function of the observation sequence $y_{[0,n]}=(y_0,\cdots,y_n)$
$$
p_\theta(y_{[0,n]})=\sum_{x_{[0,n]}\in S^{n+1}}\nu(x_0)q_{\beta_{x_0}}(y_0)\prod_{k=1}^n p_{x_{k-1},x_k} q_{\beta_{x_k}}(y_k)
$$
being a sum of $m^{n+1}$-terms, is very difficult to compute as function of $\theta$ for big $n$ (though, given a fixed $\theta$, this can be calculated by Baum-Welch algorithm in $O(m^{2}n)$-steps). We recall that the introduction of the predictor $p_k(\cdot):=\pp_{\theta,\nu}(X_{k}=\cdot|y_{[0,k-1]})$  (and the associated recursion formula) has played a crucial role in the theoretical probabilistic study of HMMs (\cite{L1992},\cite{GM2000},\cite{HW2011}), through the formula
$$
p_\theta(y_{[0,n]})=\prod_{k=0}^n \left(\sum_{i=1}^m p_k(i) q_{\beta_i}(y_k)\right).
$$

The Baum-Welch EM algorithm is very consuming in time for large $n$ and big $m$: in each iteration, one requires $O(m^2n)$-operations in the expectation step without counting the maximization step (in the mixed Gaussian or Poissonian signals cases there is the very useful re-estimation explicit formula of Baum). But when $n$ is not big enough, the problem of local minima arises. Even choosing numerous different initial points of $\theta$ could help finding the MLE, that does not work surly mathematically and that will increase in many times the computation cost.

A first concrete application of HMM  with big $n$ ($n=1000$), up to our knowledge, was carried out by Titsias, Holmes and Yau \cite{THY2016}(2016): they proposed a new algorithm based on the $k$-segment approximation method. Their method works well when $p_{ii}$ are close to $1$: in that case the number of changes of states $c_{x_{[0,n]}}=\sum_{k=1}^n 1_{x_{k-1}\neq x_k}$ is not big (say $\le k$), and
$$\pp(Y_{[0,n]}=y_{[0,n]}, c_{X_{[0,n]}}\le k)$$ being close to the likelihood function $p_\theta(y_{[0,n]})$, is a sum of at most $\sum_{j=1}^k C_{n+1}^j m^2$-terms, much more easier to treat for not big $k$. They developed the algorithms associated with this $k$-segment approximation.
However their innovative method losses its pertinence when the probability that $c_{X_{[0,n]}}>k$ is not negligible for relatively big $k$ (i.e. when $p_{ii}$ are not close to $1$).

Our motivation is: when $n$ is very large or $m$ is big such as in DNA sequencing or economics or finances and when the MLE becomes difficult to compute,  we should find some substituter of the MLE for the statistical inferences of HMMs.

\subsection{Objective}
As a substituter of the MLE, we propose the maximal entropy estimator (MEE in short, denoted by $\hat \theta^{ME}_n$) based only on the 2-dimensional empirical distribution of the observation sequence
\begin{equation}\label{Ly2}
L_n^{y,2} = \frac 1n \sum_{k=1}^n \delta_{(y_{k-1},y_k)},\ \text{($\delta_\cdot$ is the Dirac measure at the point $\cdot$)}
\end{equation}
and an algorithm for computing $\hat \theta^{ME}_n$. Our method works for large observation dataset (big $m, n$) and for several statistical purposes such as parameter estimation or hypotheses testing, by showing the identifiability, the consistency and the asymptotic normality of $\hat \theta^{ME}_n$ or convergence in law of the 2-dimensional relative entropy.

Our starting point is a very naive feeling: as $L_n^{Y,2}$ converges in law to the stationary 2-dimensional distribution $Q^{Y,2}_\theta$ of $(Y_0,Y_1)$, it is stable  (varying few randomly) and robust (depending few on the possible errors in the observations $(y_{k-1},y_k)$ for some $k$), unlike the very random sequence $Y_{[0,n]}$. Moreover $L_n^{Y,2}$ would be a sufficient statistic if $(Y_n)$ were Markov (though it is NOT).
If $Q^{Y,2}_\theta$ determines uniquely $\theta$,  $L_n^{Y,2}$ would become an {\it asymptotically sufficient} statistic. That will allow us to reduce the statistical problems of sample size $n+1$ of HMMs to dimension 2. The main objective of this paper is to rend the above naive intuition rigorous and useful for statistical inferences.

\subsection{Organization}
In the next section 2, we show the first crucial theoretical result which says that the 2-dimensional stationary distribution $Q_\theta^{Y,2}$ determines uniquely all unknown parameters in $\theta$ (the so called identifiability), as for stationary Markov chains (whereas $(Y_n)$ is not Markov). That justifies rigorously our naive intuition above: as $L_n^{Y,2}\to Q_\theta^{Y,2}$, $L_n^{Y,2}$ distinguishes or determines uniquely $\theta$ if $n$ is big enough, i.e. it is an {\it asymptotically sufficient} statistic.
The
MEE $\hat \theta^{ME}_n$ and the associated gradient descent algorithm  for HMMs are presented in Section 3.  We prove the strong consistency and the asymptotic normality  of the MEE,  and provide the explicit expression of the asymptotic covariance matrix based on the 2-dimensional Fisher's information in Section 4. The hypthesis testing results based on the 2-dimensional relative entropy, including  type I error  and type II error estimates, are given in Section 5. In Section 6 we discuss HMMs with signals of mixture of exponential model, covering the usual Gaussian, Poisson cases. In Section 7, we furnish
numerical simulations and statistical analysis of several examples for illustrating the usefulness  of MEE and and the numerical validity of the 2RE algorithm.

\section{Assumptions and the identifiability}

\subsection{Notations}
At first the signal space $S$ is either at most countable or the Euclidean space $\rr^l$ equipped with the discrete metric $d(y,y')=1_{y\ne y'}$ or the Euclidean metric $d(y,y')=|y-y'|$, and the associated Borel $\sigma$-field $\SS$. On the space $\MM_1(S)$ of probability measures on $(S,\SS)$, besides the weak convergence topology, we recall the total variational metric between $\mu,\nu\in \MM_1(S)$
$$
\|\nu-\mu\|_{tv}=\sup_{A\in\SS}|\nu(A)-\mu(A)|.
$$
Its probabilistic meaning is
$$
\|\nu-\mu\|_{tv}=\inf_{X,Y} \pp(X\ne Y)
$$
where the infimum is taken over all couples of random variables $X,Y$ so that the law of $X$ (resp. Y) is $\mu$ (resp. $\nu$) (a such couple $(X,Y)$ is called a {\it coupling of $(\mu,\nu)$}).

Given $\theta\in \Theta$ and an initial distribution $\nu$ of the hidden Markov chain, we denote by $\pp_{\theta, \nu}$ the probability measure on $(\Omega,\FF)$ under which
$(Z_n=(X_n,Y_n))_{n\ge0}$ is the HMM with all parameters given in the Introduction, so that the law of $X_0$ is $\nu$.

\subsection{Assumptions}
Throughout the paper we assume that for our HMM,

\medskip
{\bf (H0)} {\it For the vector of the unknown parameters $\theta=(\theta_1,\cdots,\theta_M)$,
	\begin{itemize}
		\item $\theta$ varies in the closure $\Theta$ of some bounded and open subset $\Theta^0$ of $\rr^M$.
		
		\item For all $i\in\mm,\theta\in\Theta$, $\beta_i(\theta)\in \bar O$, the closure of some open and bounded subset $O$ of $\rr^d$.
		
		\item
		The mapping $\theta\to ((p_{ij}(\theta))_{i,j\in\mm}, (\beta_i(\theta))_{i\in\bar O})$ is continuous and injective on $\Theta$.
		
		\item $\beta\to q_{\beta}$ is a continuous mapping from $\Theta$ to $\MM_1(S)$ equipped with the weak convergence topology.
	\end{itemize}
}

Our next assumption is about the ergodicity and the aperiodicity of the hidden Markov chain $(X_k)_{k\ge 0}$.

\medskip
{\bf (H1)} {\it There are $n_0\in \nn^*$, $\kappa>0$ and  a probability measure $\nu_0$ on $\mm$ charging all states of $\mm$ such that for any $\theta\in \Theta$, the transition probability matrix $P_\theta=(p_{ij}(\theta))$ satisfies : $P_\theta^{n_0}(i,j)\ge \kappa \nu_0(j)$ for all $i,j\in \mm$.}
\medskip

For the signal distributions $q_{\beta_i(\theta)}$, $1\le i\le m$, we assume

{\it
	\begin{itemize}
		\item[\bf (H2)]{\bf (the hidden states are ordered by $(\beta_i=\beta_i(\theta))$)}  For any re-ordering $\tau: \mm\longrightarrow \mm$ (bijection), if $$(\beta_{\tau(1)}, \cdots, \beta_{\tau(m)}) = (\beta_1, \cdots, \beta_m),$$ then $\tau(i)=i$, $\forall i \in \mm$. In other words the hidden states are ordered by the parameters $(\beta_1,\cdots,\beta_m)$. This implies that $\beta_i\not=\beta_j$ for different hidden states $i,j$.
		\item[\bf (H3)] {\bf (the identifiability of the hidden states)} For $\beta_i, \tilde \beta_i\in \bar O, i=1,\cdots, m$, if
		$$
		\sum_{i=1}^m c_i q_{\beta_i} =  \sum_{i=1}^m \tilde c_i q_{\tilde \beta_i}
		$$
		where $c_i, \tilde c_i\ge0$ and $\sum_i c_i=\sum_i\tilde c_i=1$, then $\sum_i c_i \delta_{\beta_i}= \sum_i \tilde c_i \delta_{\tilde \beta_i}$.

	\end{itemize}
}

\begin{rem}{\rm  The reader is referred to the known works \cite{BK1957}, \cite{Teicher1963}, \cite{FR1992}, \cite{R1996-2}, \cite{BB1998} (and the references therein) for HMMs satisfying the identifiability {\bf (H3)} of the hidden states.}\end{rem}

Those four assumptions will be assumed throughout the paper.

\subsection{2-dimensional contiguous empirical distribution}
Under {\bf (H1)}, $(X_k)$ has a unique invariant probability measure $\mu_\theta$ on $\mm$, i.e. $\mu_\theta P_\theta=\mu_\theta$ when $\mu_\theta$ is identified as the line-vector $(\mu_\theta(i))_{1\le i\le m}$, and $\mu_\theta(i)\ge \kappa \nu_0(i)>0$ for each $i\in \mm$. Moreover applying the classic Doeblin's theorem (which is a quantitative refinement of the famous Perron-Frobenius theorem) to the Markov chain $Z_n=(X_n,Y_n)$, we have for any initial distribution $\nu$ of $X_0$ and for every measurable subset $A$ of $\mm\times S$,
\begin{equation}\label{a21}
|\pp_{\theta, \nu}((X_n,Y_n)\in A) - \sum_{i=1}^m\mu_\theta(i) q_{\beta_i}\{s; (i,s)\in A\}|\le (1-\kappa)^{[n/n_0]}, \ \forall n\ge0.
\end{equation}
In particular as $n$ goes to infinity, the 2-dimensional (contiguous) empirical measures
\begin{equation}
L_n^{Y,2}=\frac 1n \sum_{k=1}^n \delta_{(Y_{k-1}, Y_k)}
\end{equation}
converges $\pp_{\theta,\nu}$-a.s. in the weak convergence topology, to the stationary 2-dimensional (contiguous) observation distribution $Q_\theta^{Y,2}$ determined  by
\begin{equation}
Q_\theta^{Y, 2}(A_0\times A_1)=\sum_{i, j \in \mm}\mu_\theta(i)p_{ij}(\theta)q_{\beta_i(\theta)}(A_0)q_{\beta_{j}(\theta)   }(A_1), \ A_0,A_1\in \SS.
\end{equation}
If we could prove that $Q_{\theta}^{Y,2}$ determines uniquely $\theta$, $L_n^{Y,2}$ would become an asymptotically sufficient statistic of $\theta$. That is the purpose of the following

\begin{prop}[Identifiability]\label{prop: iden}
Under {\bf (H0), (H1), (H2) and (H3)}, the unknown parameter $\theta$ is identifiable via the 2-dimensional distribution $Q_\theta^{Y, 2}$, i.e. for $\theta^0,\theta^1\in\Theta$, if $Q_{\theta^0}^{Y, 2}=Q_{\theta^1}^{Y, 2}$, then $\theta^0=\theta^1$.

Furthermore if $Q_{\theta_n}^{Y, 2}\to Q_{\theta}^{Y, 2}$, then $\theta_n\to \theta$.

\end{prop}
\begin{proof} At first
the 1-dimensional stationary distribution $$Q_{\theta}^{Y,1}(dy)=\pp_{\theta,\mu_\theta}(Y_0\in dy)=\sum_{i\in \mm}\mu_{\theta}(i)q_{\beta_i(\theta)}(dy)$$ is finite mixture of $\{q_\beta; \beta \in \bar O\}$. Given $\theta^0,\theta^1\in\Theta$, if $Q_{\theta^0}^{Y, 2}=Q_{\theta^1}^{Y, 2}$, then $Q_{\theta^0}^{Y,1}=Q_{\theta^1}^{Y,1}$. By the identifiability of hidden states in {\bf (H3)},
$$
\sum_{i=1}^m \mu_{\theta^0}(i) \delta_{\beta_{i}(\theta^0)} = \sum_{i=1}^m \mu_{\theta^1}(i) \delta_{\beta_{i}(\theta^1)}.
$$
Therefore $\beta_{i}(\theta^0)=\beta_{i}(\theta^1)$ for all $i$ by {\bf (H2)} and then $\mu_{\theta^0}=\mu_{\theta^1}$. Below we can write $\beta_i=\beta_{i}(\theta^k)$ and $\mu_\theta=\mu_{\theta^k}$ for $k=0,1$.

We turn now to the identification of the transition probabilities $p_{ij}$.   Since
$$
Q_{\theta^0}^{Y, 2}=\sum_{i, j \in \mm}\mu_\theta(i)p_{ij}(\theta^{0})q_{\beta_i}\otimes q_{\beta_j}=Q_{\theta^1}^{Y, 2}=\sum_{i, j \in \mm}\mu_\theta(i)p_{ij}(\theta^{1})q_{\beta_i}\otimes q_{\beta_j}
$$
we have any $A_0\in \SS$,
$$
\sum_{j=1}^m \left(\sum_{i=1}^m\mu_\theta(i)p_{ij}(\theta^{0})q_{\beta_i}(A_0)\right) q_{\beta_j}=\sum_{j=1}^m \left(\sum_{i=1}^m\mu_\theta(i)p_{ij}(\theta^{1})q_{\beta_i}(A_0)\right) q_{\beta_j}
$$
Taking the value of $S$ of those two measures, we obtain
$$
\sum_{i,j}\mu_\theta(i)p_{ij}(\theta^{0})q_{\beta_i}(A_0)=\sum_{i,j}\mu_\theta(i)p_{ij}(\theta^{1})q_{\beta_i}(A_0).
$$
Hence whenever this sum is not zero, we obtain by {\bf (H2) and (H3)}
$$
\sum_{i=1}^m\mu_\theta(i)p_{ij}(\theta^{0})q_{\beta_i}(A_0)=\sum_{i=1}^m\mu_\theta(i)p_{ij}(\theta^{1})q_{\beta_i}(A_0),\ \forall j
$$
which still holds true if the sum above is zero. As $A_0$ is arbitrary, the above equality holds in the measure sense. Using once more {\bf (H2) and (H3)}
we obtain
$$\mu_\theta(i)p_{ij}(\theta^{0})=\mu_\theta(i)p_{ij}(\theta^{1}), \ \forall i,j$$ i.e. $p_{ij}(\theta^{0})=p_{ij}(\theta^{1})$ for $\mu_\theta(i)>0$ for all $i$.

In summary we have proved $\beta_{i}(\theta^0)=\beta_{i}(\theta^1)$ and $p_{ij}(\theta^0)=p_{ij}(\theta^1)$ for all $i,j\in \mm$.
Then $\theta^{0}=\theta^{1}$ by the injectivity in {\bf (H0)}.

For the last claim, it is enough to show that the inverse mapping $\Phi^{-1}$ of $\Phi: \Theta\to F$ is continuous where $\Phi(\theta)=Q^{Y,2}_\theta$, $F=\{Q^{Y,2}_\theta;\theta\in \Theta\}$. As $\Phi$ is continuous and injective (just proved above), $\Phi$ send every compact
subset of $\Theta$ to a compact subset of $F$. But since $\Theta$ is compact by our assumptions, every closed subset of $\Theta$ is compact. Then $\Phi^{-1}:F\to \Theta$ is continuous. That finishes the proof.
\end{proof}

\subsection{Two classical examples}
\begin{example}[HMM with Gaussian observations] {\rm This is the most used HMM. The signal space $S$ is $\rr$, and for $\beta=(m, 1/(2\sigma^2))$,
		$$
		q_{\beta}=\NN(m,\sigma^2)=\frac 1{\sqrt{2\pi}\sigma} \exp\left(-\frac{(y-m)^2}{2\sigma^2}\right)
		$$
		the normal law with mean $m$ and variance $\sigma^2>0$. We take $\theta=((p_{ij})_{i\ne j}, (\beta_i))$, where
		$$
		\beta_i=(m_i, 1/(2\sigma_i^2)),\ i\in \mm.
		$$
		In other words we assume all transition probabilities $p_{ij}$ and $(m_i, \sigma_i^2)$ in the distribution of signal of state $i$ are unknown. {\bf (H0)} is obviously satisfied.
		
		On $\rr^2$ consider the total order $(x_1, y_1)\prec (x_2,y_2)$ defined by $x_1<x_2$ or $(x_1=x_2, y_1<y_2)$ (lexicographical order) . This total order allows us to orderer the hidden states in $\mm=\{1,\cdots,m\}$ by
		$$
		\beta_1\prec\beta_2\prec\cdots\prec\beta_m.
		$$
		With this ordering of the hidden states, {\bf (H2)} is satisfied. To verify other assumptions, we must specify the domains $\Theta^h, \Theta^s$ where our unknown parameters $\theta^h=(p_{ij})_{i\ne j}, \theta^s=(\beta_i)_{i\in\mm}$ vary. We assume
		$$
		-\frac 1\delta\le m_i \le \frac 1\delta, \ \ \delta\le \sigma_i^2\le \frac 1\delta, \ \  |m_i-m_j| + |\sigma_i^2-\sigma_j^2|\ge \delta
		$$
		for some sufficiently small $\delta$ given {\it a priori} (the last condition means that the signals emitted by different hidden states are sufficiently different). Finally $\Theta^s$ is the set of all $(\theta^s_i=(m_i, 1/(2\sigma_i^2)))_{i\in\mm}$ satisfying the two conditions above.
		
		We assume that for some $n_0\in \nn^*$, $\kappa>0$ and $\nu_0\in \MM_1(\mm)$,
		$$\Theta^h=\{\theta^h=(p_{ij})_{i\ne j}: \ p_{ij}\ge0, p_{ii}:=1-\sum_{j: j\ne i}p_{ij}\ge0,\ P^{n_0}(i,j)\ge \kappa \nu_0(i), \forall i,j\}$$
		where $P(i,j)=p_{ij}$ for $i\ne j$ and $P(i,i)=1-\sum_{j\ne i} p_{ij}$.
		
		Finally $\Theta=\Theta^h\times \Theta^s$. With the choice of $\Theta^h$, we see that {\bf (H1)} is satisfied.
		
		The identifiability of the hidden states in {\bf (H3)} is well known (e.g. \cite{Teicher1963}).
		
}\end{example}

\begin{example}[HMM with Poisson observations] {\rm In this example $S=\nn$ and
		$$
		q_\beta(k)=e^{-\beta} \frac {\beta^k}{k!}, \ k\in\nn.
		$$
		We order (or name) the states of $\mm$ by
		$$
		0<\beta_1<\beta_2<\cdots<\beta_m
		$$
		and assume that for some $\delta>0$ sufficiently small,
		$$
		\delta\le\beta_i\le \frac 1\delta,\  \beta_{i+1}-\beta_i\ge \delta
		$$
		for all $i$. The domain of $\theta^s=(\beta_1,\cdots,\beta_m)$ is the closed and bounded set satisfying those two conditions.
		
		Taking $\theta=(p_{ij})_{i\ne j}$ and $\Theta^h$ as in the previous example and $\Theta=\Theta^h\times \Theta^s$, we see that assumptions  {\bf (H0)} ,  {\bf (H1)} and  {\bf (H2)}   are satisfied, and the identifiability of the hidden states {\bf (H3)} is also well known (\cite{Teicher1963}) .
		
	}
\end{example}

See  \cite{BK1957}, \cite{Teicher1963}, \cite{FR1992}, \cite{R1996-2}, \cite{BB1998} for much more examples for which ${\bf (H3)}$ is satisfied.

\section{ MEE and the associated gradient descent algorithm }

\subsection{Maximum entropy estimator (MEE)}
We firstly recall the general relative entropy (also called Kullback-Leibler divergence or information).
\begin{defn}
Let $\mu$ and $\nu$ are two probability measures on the same general measurable space $(E, \EE)$. The relative entropy $H(\nu|\mu)$ of $\nu$ w.r.t. $\mu$ is defined by
$$H(\nu|\mu)=\left\{
\begin{array}{cc}
\int \log \frac {d\nu}{d\mu}d\nu, & \text{if } \nu \ll \mu \\
+\infty,  & \text{otherwise }. \\
\end{array}\right.$$
\end{defn}
\noindent
($\log x=\log_e x$).
Specifically, for two discrete distribution  $\nu\ll \mu$ on the at most countable set $S$,

$$H(\nu|\mu)=\sum_{y\in S: \nu(y)>0}\nu(y)\log\frac{\nu(y)}{\mu(y)}.$$

As the identifiability of $Q_{\theta}^{Y, 2}$ holds in the quite general framework {\bf (Hk)} ($k=0,1,2,3$), then $L_n^{Y,2}$ is an asymptotically sufficient statistic of $\theta$.
The principle of maximum entropy in statistical mechanics suggests that the true probability distribution $Q_{\theta_0}^{Y, 2}$ should minimize the relative  entropy $H(L_n^{Y,2}|Q_{\theta}^{Y, 2})$, i.e. maximize the entropy in physics because the relative entropy is un constant minus the entropy in physics.  This makes sense only if $S$ at most countable. The relative entropy is a crucial tool both in probability, statistics and information.

\begin{defn}[MEE] {\it When $S$ is at most countable and given the observation sequence $Y_{[0,n]}=y_{[0,n]}$, the (2-dimensional) maximum entropy estimator $\theta^{ME}_n$ of $\theta$ is defined as
	\begin{equation}
	\begin{array}{ll}
		\theta^{ME}_n&=\arg\min_{\theta\in \Theta} H(L_n^{Y,2}|Q_{\theta}^{Y, 2})\\
	\end{array}
	\end{equation}
}
\end{defn}
More generally  let $Q_{\theta}^{Y, 2}(y,y')$ be the density of $Q_{\theta}^{Y, 2}$ w.r.t. $\sigma(dy)\sigma(dy')$.
As
$$
\begin{array}{ll}
H(L_n^{Y,2}|Q_{\theta}^{Y, 2})&= \sum_{(y,y')\in S^2}  L_n^{Y,2}(y,y') \log L_n^{Y,2}(y,y') - \sum_{(y,y')\in S^2}  L_n^{Y,2}(y,y') \log Q_{\theta}^{Y,2}(y,y')\\
&=\sum_{(y,y')\in S^2}  L_n^{Y,2}(y,y') \log L_n^{Y,2}(y,y') - \frac 1n \sum_{k=1}^n   \log Q_{\theta}^{Y,2}(y_{k-1},y_k)\\
\end{array}
$$
where the first term in the last line above does not depend on $\theta$, and the second term makes sense even in the continuous signal case. That is why
the MEE can be defined by

\begin{defn}{\it In general signal space case, given the observation signal sequence $Y_{[0,n]}=y_{[0,n]}$, the MEE $\theta_n^{ME}$ is defined as
	\begin{equation}
	\theta^{ME}_n=\arg\min_{\theta\in \Theta}   - \frac 1n \sum_{k=1}^n   \log Q_{\theta}^{Y,2}(y_{k-1},y_k).
	\end{equation}
}
\end{defn}

Given the observation sequence $(y_0, \cdots, y_n)$, we will use $\hat{\theta}^{ME}_n$ for estimation of the unknown parameters. Before doing that we first introduce an algorithm for finding the MEE.

\subsection{Gradient descent algorithm for MEE}
For finding the minimum of
\begin{equation}\label{Htheta}
H(\theta)=-\frac 1n \sum_{k=1}^n   \log Q_{\theta}^{Y,2}(y_{k-1},y_k)\end{equation}
the ideal mathematical and also the most applied approach is to consider the gradient
flow associated to the objective function $H(\theta)$
\begin{equation}\label{gf}
\frac d{dt} \theta(t)= -\nabla_\theta H(\theta(t)).
\end{equation}
Gradient descent algorithm is just the Euler method for solving this differential equation, described as follows. Given the observation sequence $Y_{[0,n]}=y_{[0,n]}=(y_0, \cdots, y_n)$, gradient descent algorithm for MEE can be reformulated as follows in the case where $S$ is finite: choose a suitable small step size $\epsilon>0$,



\begin{algorithm}[!hbt]\label{algo1}
	\caption{Gradient descent algorithm for MEE with finite $S$}
	\KwIn{The observation sequence $Y_{[0,n]}=y_{[0,n]}=(y_0, \cdots, y_n)$, step size $\epsilon>0$, an arbitrary initial point $\theta{(0)}$, and some fixed error $\delta>0$}
	\KwOut{$\hat{\theta}^{ME}_n$}
	Calculate the 2-dimensional empirical distribution $L_n^{Y,2}(y, y')$, $(y, y')\in S^2$.\\
	\Repeat{$H(L^{Y,2}_n|Q^{Y,2}_{\theta(k)})$ is less than $\delta$}{ The $(k+1)$th iteration is
		\begin{equation}
		\theta{(k+1)}=\theta{(k)}+\epsilon\cdot\sum_{(y, y') \in S^2}L_n^{Y,2}(y, y') \nabla_\theta \log Q_{\theta}^{Y, 2}(y, y')|_{\theta=\theta(k)}.\end{equation}\\}
	
\end{algorithm}

The repeat step works when $|S|^2$ is not too big. When $S$ is continuous, we
have two choice: the first one is to discretize the continuous distribution and use the same way in discrete case, the second is to
change the repeat step in the algorithm by
\begin{equation}
\theta{(k+1)}=\theta{(k)}+\epsilon\cdot\frac 1n\sum_{k=1}^n \nabla_\theta \log Q_{\theta}^{Y, 2}(y_{k-1}, y_k)|_{\theta=\theta{(k)}}.
\end{equation}
The calculation of $\nabla_\theta Q_{\theta}^{Y, 2}$ requires the derivatives of the invariant measure $\mu_\theta$ which is only implicitly depending on $P_\theta=(p_{ij}(\theta))_{i,j\in\mm}$. Fortunately, it can be calculated explicitly in full generality.

\begin{prop}\label{prop: pa} Assume {\bf (H1)}, for any initial distribution $\nu$ on $\mm$,
\begin{equation}\label{prop: pa1}
\|\nu P_\theta^n - \mu_{\theta}\|_{tv} \le (1-\kappa)^{[n/n_0]}, n\in\nn.
\end{equation}
Moreover if $\theta\to P_\theta$ is $C^1$-smooth, then for any $l=1,\cdots, M$,
\begin{equation}\label{prop: pa2}
\partial_{\theta_{l}}\mu_\theta=\sum_{k=0}^{+\infty}(\mu_\theta \cdot \partial_{\theta_{l}}P_\theta)\cdot P_\theta^k\end{equation}
\end{prop}
Those two formulas are applied for the computation of $\mu_\theta$ and  $\partial_{\theta_{l}}\mu_\theta$ in the algorithm above.

\begin{proof} The first explicit geometric convergence of $\nu P_\theta^n$ to $\mu_\theta$ is a direct consequence of Doeblin's theorem. For the second conclusion, taking derivative w.r.t. $\theta_l$ in $\mu_\theta=\mu_\theta P_\theta$, we obtain
$$
\partial_{\theta_l} \mu_\theta = (\partial_{\theta_l} \mu_\theta) P_\theta + \mu_\theta \partial_{\theta_l}P_\theta.
$$
If $\sum_i \nu(i)=0$, by (\ref{prop: pa1}),
\begin{equation}
\sum_j|(\nu P^n)(j)|\le (1-\kappa)^{[n/n_0]} \sum_i|\nu(i)|.
\end{equation} As
$$
\sum_j (\mu_\theta \partial_{\theta_l}P_\theta)(j)=\sum_i \mu_\theta(i) \partial_{\theta_l} \left(\sum_j p_{ij}(\theta)\right)=0
$$
then 
$$\|\mu_{\theta}\partial_{\theta_l} P_{\theta} P^k_{\theta}\|_{tv}\le (1-\kappa)^{[k/{n_0}]}\|\mu_{\theta}\partial_{\theta_l}P_{\theta}\|_{tv}.$$
Moreover
$$\|\partial_{\theta_l}P_{\theta}\|_{tv}=\frac12\sum_j \left|\sum_i \nu(i)\partial_{\theta_{l}} p_{ij}(\theta)\right|\le \frac12\max_{i}\sum_j |\partial_{\theta_{l}}p_{ij}(\theta)|.
$$
Therefore
$$\partial_{\theta_l} \mu_\theta =\sum_{k=0}^\infty\mu_\theta (\partial_{\theta_l}P_\theta) P_\theta^k
$$
and the series is geometrically convergent:
$$\|\mu_{\theta}(\partial_{\theta_l}P_{\theta})P^k_{\theta}\|_{tv}\le \frac12 \max_i \sum_j |\partial_{\theta_l} p_{ij}(\theta)|(1-\kappa)^{[k/n_0]}.$$
\end{proof}

\section{The strong consistency and the asymptotic normality of the MEE}
\subsection{Strong consistency}

At first we introduce

\medskip\noindent
{\bf (H4)}  {\it the signal distributions $(q_\beta)_{\beta\in \bar O}$ satisfy
	\begin{itemize}
		
		\item for any $\beta, \beta'\in \bar O$, $q_\beta, q_{\beta'}$ are equivalent;
		
		\item
		$q_\beta(y)$ is continuous in $\beta\in \bar O$ for $\sigma$-a.e. $y\in S$;
		
		\item for any $\beta\in \bar O$,
		$$
		\sup_{\beta'}\left|\log \frac{q_\beta(y)}{q_{\beta'}(y)}\right| \in L^1(q_\beta).
		$$
	\end{itemize}
}
\medskip
\begin{thm}\label{thm: sc}{\bf (Strong consistency)} Assume (H0-H4). For any $\theta\in \Theta$ and any initial distribution $\nu$ of $X_0$, we have under $\pp_{\theta,\nu}$,
$$\hat{\theta}^{ME}_n \rightarrow \theta,  \text{a.s.}$$
\end{thm}

\begin{proof}[I. Proof in the finite-signal case without {\bf (H4)}] Let $S$ be finite. We have $\pp_{\theta,\nu}$-a.s.
$$
0\le H(L_n^{Y,2}|Q_{\theta^{ME}_n}^{Y,2})\le H(L_n^{Y,2}|Q_{\theta}^{Y,2})\to 0
$$
by the law of large number for the geometrically ergodic Markov chain $Z_n=(X_n,Y_n)$. By Csiszar-Kullback-Pinsker's inequality
$$
\|\nu-\mu\|_{tv}^2\le \frac 12H(\nu|\mu)
$$
we have
$$
\|L_n^{Y,2}-Q_{\theta^{ME}_n}^{Y,2}\|_{tv}\to 0,\  \|L_n^{Y,2}-Q_{\theta}^{Y,2}\|_{tv}\to 0,\ \pp_{\theta,\nu}-a.s.
$$
Therefore $\|Q_{\theta^{ME}_n}^{Y,2}- Q_{\theta}^{Y,2}\|_{tv}\to 0, \pp_{\theta,\nu}-a.s.$. By Proposition \ref{prop: iden}, $\theta^{ME}_n\to \theta$, $\pp_{\theta,\nu}-a.s..$
\end{proof}

\begin{proof}[II. Proof in the general signal case] The proof becomes much more difficult.

For any $\varepsilon>0$, let $A_n=[|\theta^{ME}_n-\theta|\ge \varepsilon]$. We remark that for $H(\theta)$ given in (\ref{Htheta}),
	
\begin{equation}\label{Tsc1}
\begin{split}
0&\le H(\theta)-H(\theta_n^{ME})= - \frac 1n \sum_{k=1}^n \log \frac{Q^{Y,2}_\theta(Y_{k-1},Y_k)}{Q^{Y,2}_{\theta^{ME}_n}(Y_{k-1},Y_k)}
\\
&\le - 1_{A_n}\inf_{\theta' :|\theta'-\theta|\ge \varepsilon}\frac 1n \sum_{k=1}^n \log \frac{Q^{Y,2}_\theta(Y_{k-1},Y_k)}{Q^{Y,2}_{\theta'}(Y_{k-1},Y_k)}-1_{A_n^c} \inf_{\theta' :|\theta'-\theta|< \varepsilon}\frac 1n \sum_{k=1}^n \log \frac{Q^{Y,2}_\theta(Y_{k-1},Y_k)}{Q^{Y,2}_{\theta'}(Y_{k-1},Y_k)}
\end{split}
\end{equation}

Let
\begin{equation}\label{Tsc2}
h(y,y'):= \left(\theta'\to h(y,y')(\theta')=\log \frac{Q^{Y,2}_\theta(y,y')}{Q^{Y,2}_{\theta'}(y,y')}\right)_{\theta'\in\Theta}
\end{equation}
which is valued in the separable Banach space $C(\Theta)$ of continuous functions on the compact $\Theta$ equipped with sup-norm $\|\cdot\|_{\Theta}$.
Let us admit that $\|h\|_\Theta$ is $Q^{Y,2}_\theta$-integrable, whose proof, quite technical, is left in the Appendix.
By the Banach space valued version of the ergodic theorem, we have $\pp_{\theta,\nu}$-a.s.
$$
\sup_{\theta'\in\Theta}\left|\frac 1n \sum_{k=1}^n h(Y_{k-1},Y_k)(\theta') - \int_{S^2} h(y,y')(\theta')  dQ^{Y,2}_\theta\right|\to 0.
$$
Therefore $\pp_{\theta,\nu}$-a.s.
\begin{eqnarray*}
&\inf_{\theta' :|\theta'-\theta|\ge \varepsilon}\frac 1n \sum_{k=1}^n \log \frac{Q^{Y,2}_\theta(Y_{k-1},Y_k)}{Q^{Y,2}_{\theta'}(Y_{k-1},Y_k)}\to \inf_{\theta' :|\theta'-\theta|\ge \varepsilon} H(Q^{Y,2}_\theta|Q^{Y,2}_{\theta'})=c(\varepsilon)\\
&\inf_{\theta' :|\theta'-\theta|<\varepsilon}\frac 1n \sum_{k=1}^n \log \frac{Q^{Y,2}_\theta(Y_{k-1},Y_k)}{Q^{Y,2}_{\theta'}(Y_{k-1},Y_k)}\to
\inf_{\theta' :|\theta'-\theta|< \varepsilon}H(Q^{Y,2}_\theta|Q^{Y,2}_{\theta'})=0
\end{eqnarray*}
Since $\theta'\to H(Q^{Y,2}_\theta|Q^{Y,2}_{\theta'})$ is lower semi-continuous, its infimum over the compact $\{\theta'\in\Theta : |\theta'-\theta|\ge \varepsilon\}$ is attained, so the constant above $c(\varepsilon)$ is positive. Taking $\liminf_{n\to+\infty}$ in (\ref{Tsc1}), we obtain
$$
0\le - \limsup_{n\to\infty} 1_{A_n}\cdot c(\varepsilon),\ \ \pp_{\theta,\nu}-a.s.
$$
where it follows that $\pp_{\theta,\nu}(A_n,i.o.)=0$. That is the desired strong consistency.
\end{proof}

\subsection{Central limit theorem for MEE}

We require the Fisher information of $(Q_\theta^{Y, 2})_{\theta\in \Theta^0}$ for the CLT of $\theta^{ME}_n$. We state at first

\medskip
{\bf (H5)} {\it The model is $C^2$-regular, more precisely
	
	\begin{itemize}
		\item $\theta\to p_{ij}(\theta)$, $\theta\to \beta_i(\theta)$, $i,j\in\mm$ are $C^2$-smooth on $\Theta^0$;
		\item $\beta\to q_{\beta}(y)$ is $C^2$-smooth for $\sigma$-a.e. $y\in S$;
		\item $\sup_{\beta'\in K}|\nabla_{\beta'} \log q_{\beta'}| \in L^2(q_\beta)$ for all compact subset $K\subset O$ and $\beta\in O$;
		\item $\sup_{\beta'\in K}\|\nabla^2_{\beta'} \log q_{\beta'}\| \in L^1(q_\beta)$ for all compact subset $K\subset O$ and $\beta\in O$.
		
	\end{itemize}
}
Here $\|A\|:=\sup_{|z|=1}|Az|$ is the matrix norm.

\begin{defn} {\it The Fisher-information matrix $I_{2}(\theta)$ of $(Q_\theta^{Y, 2})_{\theta\in \Theta^0}$ is defined by
	$$I_{2}(\theta)=\int_{S^2}\nabla_\theta\log Q_\theta^{Y, 2}\cdot \nabla_\theta^T\log Q_\theta^{Y, 2} dQ_\theta^{Y, 2}$$
	which is the covariance matrix $\left({\rm Cov}(\partial_{\theta_i}\log Q_\theta^{Y, 2}, \partial_{\theta_j}\log Q_\theta^{Y, 2})\right)_{1\le i,j\le M}$ under $ Q_\theta^{Y, 2}$.}  Here $Q_{\theta}^{Y,2}$ is also interpreted as the density w.r.t. $\sigma(dy)\sigma(dy')$.
\end{defn}

\begin{thm}\label{thm-CLT} Assume (H0-H5). Suppose that
\begin{equation}
I_{2}(\theta) \text{ is invertible for all }\theta\in\Theta^0
\end{equation}
then for any $\theta\in\Theta^0$ and any initial distribution $\nu$ of $X_0$, the MEE $\hat{\theta}^{ME}_n$ is asymptotically normal under $\pp_{\theta,\nu}$:

$$\sqrt{n}(\hat{\theta}^{ME}_n-\theta)\xrightarrow{\mathcal{L}}N(0, I_{2}(\theta)^{-1}\Gamma_\theta I_{2}(\theta)^{-1})$$
where  $\Gamma_\theta=(\Gamma_\theta(i,j))_{1\le i,j\le M}$ is given by
\begin{equation}\label{thm-CLT2}
\begin{split}
\Gamma_\theta(i,j)=&\cov(\partial_{\theta_i}\log Q_\theta^{Y, 2}(Y_0, Y_1), \partial_{\theta_j}\log Q_\theta^{Y, 2}(Y_0, Y_1))\\
&+\sum_{k=1}^{+\infty}\cov(\partial_{\theta_i}\log Q_\theta^{Y, 2}(Y_0, Y_1), \partial_{\theta_j}\log Q_\theta^{Y, 2}(Y_k, Y_{k+1}))\\
&+\sum_{k=1}^{+\infty}\cov(\partial_{\theta_j}\log Q_\theta^{Y, 2}(Y_0, Y_1), \partial_{\theta_i}\log Q_\theta^{Y, 2}(Y_k, Y_{k+1}))
\end{split}
\end{equation}
taken under $\pp_{\theta, \mu_\theta}$.

Moreover the asymptotic covariance matrix $I_2(\theta)^{-1}\Gamma_\theta I_2(\theta)^{-1}$ satisfies
\begin{equation}\label{CLT1}
I_2(\theta)^{-1}\Gamma_\theta I_2(\theta)^{-1}\le \left(1+\frac {2n_0}{1-\sqrt{1-\kappa}}\right)I_2(\theta)^{-1}, \ \theta\in\Theta^0.
\end{equation}
in the order of nonnegative definiteness of symmetric matrices, where $n_0, \kappa$ are given in {\bf (H1)}. 
\end{thm}

\begin{proof} As $\hat{\theta}^{ME}_n\to \theta\in \Theta^0$, $\pp_{\theta,\nu}$-a.s. by Theorem \ref{thm: sc}, $\hat{\theta}^{ME}_n\in\Theta^0$ for all $n$ big enough. Then
\begin{equation}\label{thm-CLT1}
\nabla_\theta H_n(\theta)= \nabla_\theta H_n(\theta)-  \nabla_\theta H_n(\hat{\theta}^{ME}_n)=\nabla^2_\theta H_n(\xi_n)(\theta- \hat{\theta}^{ME}_n)
\end{equation}
for some random vector $\xi_n$ in the segment from $\theta$ to $\hat{\theta}^{ME}_n$. By the strong consistency of $\hat{\theta}^{ME}_n$ again, $\xi_n\to \theta$, $\pp_{\theta, \nu}$-a.s.. For any $\delta>0$, as
\begin{align*}
&\ee_{\pp_{\theta, \mu_\theta}} \sup_{\theta': |\theta'-\theta|<\delta} \|\nabla^2_\theta H_n(\theta')- \nabla^2_\theta H_n(\theta)\|\\
 \le &\ee_{\pp_{\theta, \mu_\theta}}\frac 1n \sum_{k=1}^n \sup_{\theta': |\theta'-\theta|<\delta}\|\nabla^2_{\theta'} \log Q_{\theta'}^{Y,2}(Y_{k-1},Y_k)-\nabla^2_\theta \log Q_\theta^{Y,2}(Y_{k-1},Y_k)\|\\
=&\int \sup_{\theta': |\theta'-\theta|<\delta}\|\nabla^2_{\theta'}\log Q_{\theta'}^{Y,2}(y,y')-\nabla^2_\theta \log Q_\theta^{Y,2}(y,y')\| dQ_{\theta}^{Y,2}
\end{align*}
which tends to zero as $\delta\to 0$ by {\bf (H5)}. Then as $\delta\to 0$, $\sup_{\theta': |\theta'-\theta|<\delta} \|\nabla^2_\theta H_n(\theta')- \nabla^2_\theta H_n(\theta)\|\to 0$ uniformly in $n\ge 1$ in probability-$\pp_{\theta, \mu_\theta}$, and consequently it converges to zero in probability-$\pp_{\theta, \nu}$ ,uniformly in $n\ge 1$.
Therefore
$$
\nabla^2_\theta H_n(\xi_n)-\nabla^2_\theta H_n(\theta) \to 0
$$
in probability-$\pp_{\theta, \nu}$. Since by the ergodic theorem,
\begin{equation}\label{thm-CLT5}
\nabla^2_\theta H_n(\theta)=- \frac 1n\sum_{k=1}^n\nabla^2_\theta \log Q_\theta^{Y,2}(Y_{k-1},Y_k)\to -\int \nabla^2_\theta \log Q_\theta^{Y,2}(y_{0},y_1) dQ^{Y,2}_\theta=I_{2}(\theta)
\end{equation}
$\pp_{\theta, \nu}$-a.s., we obtain

\begin{equation}\label{thm-CLT3} \left(\nabla^2_\theta H_n(\xi_n)\right)^{-1}\to I_{2}(\theta)^{-1}, \ \text{in probability-$\pp_{\theta, \nu}$}\end{equation}
By (\ref{thm-CLT1}), on the event $A_n$ that $\theta_n^{ME}\in \Theta^0$ and $\nabla^2_\theta H_n(\xi_n)$ is invertible,

$$\aligned
\sqrt{n}(\hat{\theta}^{ME}_n-\theta)&= (\nabla^2_\theta H_n(\xi_n))^{-1} \left(\sqrt{n} \nabla_\theta H_n(\theta)\right)
\endaligned
$$
By the CLT of the Markov chain $Z_n=(X_n,Y_n)$ (Meyn-Tweedie \cite{meyn2012}) and the fact that $|\nabla_\theta \log Q_\theta^{Y,2}(y,y')|\in L^2(Q_\theta^{Y,2})$ by {\bf (H5)}, and
$$\int \nabla_\theta \log Q_\theta^{Y,2}(y,y') dQ_\theta^{Y,2}(y,y')=0$$
we have
$$
\sqrt{n}\nabla_\theta H_n(\theta)=-\frac 1{\sqrt{n}} \sum_{k=1}^n\nabla_\theta \log Q_\theta^{Y,2}(Y_{k-1},Y_k)\xrightarrow{\text{in law}} \NN(0,\Gamma_\theta)
$$
where $\Gamma_\theta$ is given by (\ref{thm-CLT2}). This, together with $\pp_{\theta,\nu}(A_n)\to 1$ and the convergence in probability in (\ref{thm-CLT3}), implies that
$$
\sqrt{n}(\hat{\theta}^{ME}_n-\theta)\xrightarrow{\text{in law}} \NN(0, I_{2}(\theta)^{-1}\Gamma_\theta I_{2}(\theta)^{-1})
$$
the desired CLT.

It remains to prove (\ref{CLT1}). For any $z\in \rr^M$ with $|z|=1$, let $z':=I_{2}(\theta)^{-1}z$

\begin{align*}
\left<z,  I_{2}(\theta)^{-1}\Gamma_\theta I_{2}(\theta)^{-1}z\right>& = \left<z', \Gamma_\theta z'\right>\\
&={\rm Var}\left(\left<z', \nabla_\theta \log Q^{Y,2}_\theta\right>(Y_0,Y_1)\right)\\
& + 2 \sum_{k=1}^{+\infty}{\rm Cov}\left(\left<z', \nabla_\theta \log Q^{Y,2}_\theta\right>(Y_0,Y_{1}), \left<z', \nabla_\theta \log Q^{Y,2}_\theta\right>(Y_k,Y_{k+1})\right)\\
&\le {\rm Var}\left(\left<z', \nabla_\theta \log Q^{Y,2}_\theta\right>(Y_0,Y_1)\right) \left(1+ 2\sum_{k=1}^{+\infty} \rho_2(P_\theta^k)\right)
\end{align*}

where 
$$\rho_2(P_\theta^k):=\sup_{\mu_\theta(f^2)\le 1} \|P_{\theta}^k f-\mu_\theta(f)\|_{L^2(\mu_\theta)}.$$
By Del Moral-Ledoux-Miclo \cite[Proposition 1.1]{del2003contraction}, 
$$
\rho_2(P_\theta^k)\le \sqrt{\frac 12\max_i \sum_j |P_\theta^k(i,j)-\mu(j)|}\le (1-\kappa)^{[k/n_0]/2}.
$$
Substituting it into the previous inequality and noting that 
$${\rm{Var}}\left(\left<z', \nabla_\theta \log Q^{Y,2}_\theta\right>(Y_0,Y_1)\right)=\left<z', I_{(2)}(\theta) z'\right>=\left<I^{-1}_{2}(\theta) z, z\right>$$
we obtain 
$$
\left<z,  I_{2}(\theta)^{-1}\Gamma_\theta I_{2}(\theta)^{-1}z\right> \le \left<z, I^{-1}_{2}(\theta) z\right>\left(1+ 2\sum_{k=1}^{+\infty} (1-\kappa)^{[k/n_0]/2}\right)
$$
where (\ref{CLT1}) follows.  
\end{proof}

\section{Hypothesis testing: $\theta=\theta_0$ vs $\theta=\theta_1$}

Of course we can use the asymptotic normality of $\theta_n^{ME}$ for the hypothesis testing: $H_0: \theta=\theta_0$ vs $H_1: \theta=\theta_1$.

Given a level of confidence $\alpha\in (1/2,1)$,
accept $H_0$, if $|\hat \theta_n^{ME}-\theta_0|\le \frac {c_\alpha}{\sqrt{n}}$ where
$$
\pp\left(|I^{-1/2}_{2}(\theta_0)\eta|<c_\alpha  \sqrt{1+\frac {2n_0}{1-\sqrt{1-\kappa}}} \right)=\alpha
$$
and $\eta$ is the standard Gaussian random vector in $\rr^M$ of law  $\NN(0,I)$.

Its level of confidence is approximatively greater than $\alpha$, by the CLT and (\ref{CLT1}) in Theorem \ref{thm-CLT}.
The second type error of this test is very small: if $|\theta_1-\theta_0|=\beta/\sqrt{n}$ with $\beta$ much bigger than $c_{\alpha}$,
$$
\aligned
\pp_{\theta_1}\left(|\hat \theta_n^{ME}-\theta_0|\le \frac {c_\alpha}{\sqrt{n}}\right)
&\le \pp_{\theta_1}\left(|\hat \theta_n^{ME}-\theta_1|\ge |\theta_1-\theta_0|-\frac {c_\alpha}{\sqrt{n}}\right)\\
&\preceq \pp_{\theta_1}\left(|I^{-1/2}_{2}(\theta_1)\eta| \ge (\beta-c_\alpha)\sqrt{1+\frac {2n_0}{1-\sqrt{1-\kappa}}} \right)
\endaligned
$$
where the law of $\eta$ is $\NN(0,I)$.

However we want to propose another hypothesis test basing on the relative entropy. The reason is very simple from the algorithm point of view: one can only approximate the MEE $\hat \theta_n^{ME}$ by $\hat \theta_{n,N}^{ME}$ when the gradient descent algorithm related to the relative entropy $H(L_n^{Y,2}|Q_{\theta}^{Y,2})$ stops at some step $N$ (when the algorithm stabilizes). So one can take $\theta_0\thickapprox\hat \theta_{n,N}^{ME}$ and to test if $\theta=\theta_0$.

Because we use the relative entropy $H(L_n^{Y,2}|Q^{Y,2}_\theta)$ for the hypothesis
testing: $\theta=\theta_0$ vs $\theta=\theta_1$, we assume that the signal space $S$ is finite in this section.

\subsection{Central limit theorem for the relative entropy and $\chi^2$-divergence}

For bounding the relative entropy, we recall the $\chi^2$-divergence $\chi^2(\nu|\mu)$ and the relationships between these metrics. We begin by recalling a known result. 

\begin{lem}\label{lem: rela}(see \cite{sason2015upper}) 
On finite state space $S$, $\chi^2$-divergence is defined by
$$\chi^2(\nu|\mu)=\sum_{y\in S}\frac{(\nu(y)-\mu(y))^2}{\mu(y)}.$$
Then we have
\begin{equation}\label{eq: rela}
2||\nu-\mu||_{tv}^2\leq H(\nu|\mu)\leq \log (\chi^2(\nu|\mu)+1)
\end{equation}
and
$$\chi^2(\nu|\mu)\leq \frac{1}{\min \mu(y)} ||\nu-\mu||_{tv}^2.$$
\end{lem}
The first inequality in (\ref{eq: rela}) is the famous Csiszar-Kullback-Pinsker inequality. 

Due to the central limit theorems of Markov chains (see Meyn and Tweedie \cite{meyn2012}), we can obtain the CLT for the relative entropy.

\begin{thm}\label{thm: clt for H} Under $\pp_{\theta,\nu}$, we have
$$
2nH(L_n^{Y, 2}|Q_\theta^{Y, 2})\xrightarrow{\text{ in law }}\sum_{y, y'}\frac{\xi^2(y, y')}{Q_\theta^{Y, 2}(y, y')}
$$
and
$$
n\chi^2(L_n^{Y,2}|Q_\theta^{Y,2})\xrightarrow{\text{ in law }}\sum_{y, y'}\frac{\xi^2(y, y')}{Q_\theta^{Y, 2}(y, y')},
$$
where $\xi=(\xi(y, y'))_{(y, y')\in S^2}$ is a centered gaussian vector with covariance matrix
\begin{eqnarray*}
	\Gamma((y, y'), (\tilde{y}, \tilde{y}'))&:=& {\rm Cov}(\xi(y, y'), \xi(\tilde y, \tilde y'))\\
	&=&\cov(1_{(y, y')}(Y_0,Y_1), 1_{(\tilde{y}, \tilde{y}')}(Y_0,Y_1))+\sum_{k=1}^{+\infty}\cov(1_{(y, y')}(Y_0, Y_1),  1_{(\tilde{y}, \tilde{y}')}(Y_{k}, Y_{k+1}))\\
	&+&\sum_{k=1}^{+\infty}\cov(1_{(\tilde{y}, \tilde{y}')}(Y_0, Y_1),  1_{(y, y')}(Y_{k}, Y_{k+1}))
\end{eqnarray*}
where the covariances in the last line above are under the stationary probability measure $\pp_{\theta, \mu_\theta}$. 
\end{thm}

\subsection{Expectation and Concentration for the limit law of the relative entropy} For applications of Theorem \ref{thm: clt for H}, we must control
the limit law  of $|\zeta|^{2}$ where
$$\zeta=\left(\zeta(y, y')=\frac{\xi(y, y')}{\sqrt{Q_\theta^{Y, 2}(y, y')}}\right)  \quad \text{ and } \quad |\zeta|^2=\sum_{(y, y')}\frac{\xi^2(y, y')}{Q_\theta^{Y, 2}(y, y')}$$
where $\xi=(\xi(y,y'))_{(y,y')\in S^{2}}$ is  given in Theorem \ref{thm: clt for H}.

\begin{prop}\label{prop55}
\begin{item}
\item (a)$$\ee_{\pp_{\theta,\nu}}|\zeta|^2\leq \left(1+\frac{2n_0}{1-\sqrt{1-\kappa}}\right)\cdot(|S^2|-1)$$

\item (b) The maximum eigenvalue $\lambda_{\max}$ of the covariance matrix $\Gamma_\zeta$  of the gaussian vector $\zeta$ satisfies
$$\lambda_{\max}\leq 1+ \frac{2n_0}{1-\sqrt{1-\kappa}}.$$

\item{(c)} For any $c>1$,
$$\pp(|\zeta|^2>c\ee|\zeta|^2)
\le\exp\left(-\frac{(\sqrt{c}-1)^{2}}{2\lambda_{\max}}\cdot\ee |\zeta|^2\right)$$
\end{item}
\end{prop}

\subsection{Hypothesis test by means of the relative entropy}
For the hypothesis testing: $H_0: \theta=\theta_0$ vs $H_1: \theta=\theta_1$,
given a level of confidence $\alpha\in (1/2,1)$, {\it accept $H_0$ if $H(L_n^{Y, 2}|Q_\theta^{Y, 2})\le \frac {c_\alpha}{n}$ (and reject $H_0$ otherwise)}. The constant $c_{\alpha}$ can be determined approximatively by the concentration of $|\zeta|^{2}$. If $n$ is big enough, the type I error
$$ \pp_{\theta_0}\left(H(L_n^{Y, 2}|Q_{\theta_0}^{Y, 2})> \frac{c_{\alpha}}{n}\right)\approx \pp_{\theta_0}\left(|\zeta|^2> 2c_{\alpha}\right).$$
One can determine $c_{\alpha}$ from Proposition \ref{prop55} so that the last probability is less than $1-\alpha$.

For estimating the second type error for big $n$, by Lemma \ref{lem: rela},
$$\aligned
\pp_{\theta_1}\left(H(L_n^{Y, 2}|Q_{\theta_0}^{Y, 2})\leq \frac{c_{\alpha}}{n}\right)
&\leq\pp_{\theta_1}\left(||L_n^{Y, 2}-Q_{\theta_0}^{Y, 2}||_{tv}\leq \sqrt{\frac{c_{\alpha}}{2n}}\right)\\
&\leq \pp_{\theta_1}\left(||L_n^{Y, 2}-Q_{\theta_1}^{Y, 2}||_{tv}\geq ||Q_{\theta_1}^{Y, 2}-Q_{\theta_0}^{Y, 2}||_{tv}-\sqrt{\frac{c_{\alpha}}{2n}}
\right)
\endaligned
$$
by the triangular inequality. Let $c(\theta_0,\theta_1):=||Q_{\theta_1}^{Y, 2}-Q_{\theta_0}^{Y, 2}||_{tv}$ which is a positive constant (by Proposition \ref{prop: iden}) and easy to be computed in practice.
By the concentration inequality of Markov chains in Djellot {\it et al.} \cite{DGW04}, once if $c(\theta_0,\theta_1) > \ee_{\theta_1}||L_n^{Y, 2}-Q_{\theta_1}^{Y, 2}||_{tv} +\sqrt{\frac{c_{\alpha}}{2n}}>0$ (the latter is an infinitesimal $O(1/\sqrt{n})$), we have
$$\aligned
 &\pp_{\theta_1}\left(||L_n^{Y, 2}-Q_{\theta_1}^{Y, 2}||_{tv}\geq ||Q_{\theta_1}^{Y, 2}-Q_{\theta_0}^{Y, 2}||_{tv}-\sqrt{\frac{c_{\alpha}}{2n}}
\right)\\
&\le \exp\left(-n\cdot \frac{2}{(1+n_0(1-\kappa)/\kappa))^2}\left[c(\theta_0,\theta_1) - \ee_{\theta_1}||L_n^{Y, 2}-Q_{\theta_1}^{Y, 2}||_{tv} -\sqrt{\frac{c_{\alpha}}{2n}} \right]^2 \right).
\endaligned
$$
The last term is exponentially small in big $n$. In other words this relative entropy hypothesis testing is exponentially powerful in $n$.

\section{Numerical examples}
In this section we illustrate the algorithm 2RE described in \S3, by means of two concrete examples with very great sample size, for which the classic algorithms are too time-consuming for being useful in practice.   We show how the algorithm 2RE perform in regard to a great size of observation sequences and how the estimated parameters differ from the true parameters. 
In the two examples there are only two hiddens states $1, 2$. 

\subsection{Example 1: Poisson observation }
The data is generated from the HMM with hidden states $\{1,2\}$ and Poisson observation sequence on $\mathbb{N}$. We suppose that the parameters in $\theta=(p_{12}, p_{21}, \beta_1, \beta_2)$ are all unknown, where $\beta_k$ is the parameter of the Poisson distribution of the signal when the hidden state is $k$ ($k=1,2$). 

We use 
$$P_{\theta}=  \left[ \begin{matrix}
0.3 & 0.7 \\
0.6 & 0.4
\end{matrix} \right]\quad
\text{and}\quad
(\beta_1, \beta_2)=(2.5, 0.5)$$
i.e. $\theta=\theta_0=(0.7, 0.6; 2.5, 0.5)$ to generate a sequence of  $n=10^5$ signals.  In this case, the Fisher information $I_{2}(\theta_0)$ of $Q_\theta^{Y, 2}$ is
$$ \left[ \begin{matrix}
0.85161298& -0.42440013&  0.19349763& -0.43938774 \\
-0.42440013&  0.94910932& -0.17431689&  0.38227075\\
 0.19349763& -0.17431689&  2.13206361& -1.0718548\\
 -0.43938774&  0.38227075& -1.0718548&   1.17119549\\
\end{matrix} \right]$$
Its inverse is

$$ \left[ \begin{matrix}
 1.75095757& 0.5470283&   0.23391698&  0.6924219 \\
0.5470283&   1.42470472& -0.11810743& -0.36788003\\
 0.23391698& -0.11810743&  0.92911967&  0.97661805\\
0.6924219&  -0.36788003&  0.97661805&  2.12745374\\
\end{matrix} \right]$$

In the algorithm, we set $\epsilon=0.001$ as the step size in the gradient descent,
and $l=30$ as the the number of iterations for calculating the invariant measure $\mu_\theta$ and its derivative $\partial_{\theta_{ij}}\mu_\theta$ (according to the formula (\ref{prop: pa2})). 

Start the 2RE algorithm with the
initial value
$$P_{\theta{(0)}}=\left[ \begin{matrix}
0.5 & 0.5 \\
0.5 & 0.5
\end{matrix} \right]\  \text{and}\ (\beta_{1}{(0)}, \beta_{2}{(0)})=(3, 0.1).$$
We obtain the iteration results in the table \ref{table2}, 

\begin{table}[!hbt]
	\centering
	\caption{Iteration results}\label{table2}
	\begin{tabular}{|c|c|c|c|}
		\hline
		$k$ & $P_{\theta{(k)}}$ & $\left(\beta_1{(k)},\beta_2{(k)}\right)$ & $H(k)$\\
		\hline
		$5000$ & $\left[ \begin{matrix}
		0.28691485 & 0.71308515 \\
		0.59653801 & 0.40346199
		\end{matrix} \right]$ & $(2.5073621,  0.50915734)$ & $5.9192171556647552 \times 10^{-4}$\\
		\hline
		$10000$ & $\left[ \begin{matrix}
		0.28788447 & 0.71211553 \\
		0.59600346 & 0.40399654
		\end{matrix} \right]$ & $(2.50811494, 0.50833996)$ & $5.923377320959681\times10^{-4}$ \\
		\hline

	\end{tabular}
\end{table}

\noindent
where $H(k)= H(L_n^{Y,2}|Q_{\theta(k)}^{Y,2})$.

As seen from the above table,  the estimated values are very precise if we have enough observations and use enough iterations in the 2RE algorithm. The mathematical reason is the Fisher information in the present Poisson signal model is not small. All numerical results above are within the prevision of our theoratical results in Theorems \ref{thm-CLT} and \ref{thm: clt for H}.

\subsection{Example 2: Gaussian observation}
The data is generated from the HMM with hidden state $\{1,2\}$ and Gaussian observation sequence on $(-\infty, \infty)$ with
$$q_{\beta}(y)=\frac{1}{\sqrt{2\pi}}\exp(-\frac{(y-\beta)^2}{2}).$$ 
We generate a signal-data sequence of length $n=5000$ of this HMM with the following parameters:
$$P_{\theta}=  \left[ \begin{matrix}
	0.2 & 0.8 \\
	0.7 & 0.3
	\end{matrix} \right] \quad \text{and} \quad (\beta_1, \beta_2)=(0, 3)$$
i.e. $\theta_0=(0.8, 0.7; 0, 3)$.  In this case, the Fisher information $I_{2}(\theta_0)$ of $Q_\theta^{Y, 2}$ is
$$ \left[ \begin{matrix}
2.0939& -0.4257&  0.2467&  0.3394\\
-0.4257&  2.0731& -0.3360& -0.2829\\
0.2467&   -0.3360&  0.7159& -0.1299\\
0.3394& -0.2829& -0.1299&  0.8537\\
\end{matrix} \right]$$
Its inverse is

$$ \left[ \begin{matrix}
0.54996962&  0.04677193& -0.21023406& -0.23513801\\
0.04677193&  0.56800595&  0.28923517&  0.21364196\\
-0.21023406&  0.28923517&  1.68409397&  0.43568218\\
-0.23513801&  0.21364196&  0.43568218&  1.40194479\\
\end{matrix} \right]$$

For the continuous signals, we replace the repeat step in the Algorithm \ref{algo1} by
$$\theta{(k+1)}=\theta{(k)}+\epsilon\cdot\frac 1n\sum_{k=1}^n \nabla_\theta \log Q_{\theta}^{Y, 2}(y_{k-1}, y_k)|_{\theta=\theta{(k)}}.$$
we set $\epsilon=0.1$ as the step size, $l=30$ as the number of iterations for calculating the invariant measure $\mu_{\theta}$ and its derivative $\partial_{\theta_{ij}}\mu_{\theta}$. Begin the algorithm 2RE with the initial value
$$ P_{\theta{(0)}}=  \left[ \begin{matrix}
	0.5 & 0.5 \\
	0.5 & 0.5
	\end{matrix} \right]\ \text{and}\ (\beta_{1}{(0)}, \beta_{2}{(0)})=(0, 1).$$
We obtain the iteration results of our 2RE in the table \ref{table1}, 

\begin{table}[!hbt]
	\centering
	\caption{Iteration results}\label{table1}
	\begin{tabular}{|c|c|c|c|}
		\hline
		$k$ & $P_{\theta{(k)}}$ & $\left(\beta_1{(k)},\beta_2{(k)}\right)$\\
		\hline
		$100$ & $\left[ \begin{matrix}
		0.19086042 & 0.80913958 \\
		0.701435 & 0.298565
		\end{matrix} \right]$ & $(0.00000604, 3.00604911)$ \\
		\hline
		$200$ & $\left[ \begin{matrix}
		0.19223189 & 0.80776811 \\
		0.70309968 &0.29690032
		\end{matrix} \right]$ & $(0.00572507, 3.01038024)$ \\
		\hline
	\end{tabular}
\end{table}

\medskip\noindent
 
For this example, the number of the observations $n=5000$ is smaller than the previous example $n=10^{6}$. Thus the estimated values are less precise. On the other hand, $\epsilon=0.1$ is bigger than $\epsilon=0.00001$ in the previous one, we get the asymptotic estimated values in much less iterations.

\section{Comparison of the MEE with the MLE}

The advantages of the MEE w.r.t. the MLE are all in the case where the sample size $n$ is big: 

\begin{enumerate}
\item when $n\ge 1000$, the EM algorithm for finding the MLE is much more time consumming than our 2RE algorithm. 

\item  Our algorithm for finding the MEE works for very big $n$, such as $n=10^5$ as in the study of the genomes; for which the EM 
algorithm is no longer practicable. 

\item Moreover  as 
$$
\nabla_\theta^2 \left(\frac 1n \sum_{k=1}^n\log Q_\theta^{Y,2}(Y_{k-1},Y_k)\right) \to \ee_{\pp_{\theta,\mu_\theta}} \nabla_\theta^2\log Q_\theta^{Y,2}(Y_{0},Y_1)=I_2(\theta), 
$$
when the Fisher information $I_2(\theta)$ is not degenerate, $\frac 1n \sum_{k=1}^n\log Q_\theta^{Y,2}(y_{k-1},y_k)$ will be strictly convex, there is no problem of local minima in our 2RE algorithm.    
\end{enumerate} 

The disadvantage of the MEE w.r.t. the MLE is:  the MLE, once computable, is the best choice because it is asymptotically efficient in the sense of Lehman. When $100\le n\le 1000$, the MLE is computable by the EM algorithm, it is better than the MEE. Notice also that the Fisher information $I(\theta)$ in the CLT of 
the MLE  (\cite{L1992},\cite{R1994},\cite{R1995-1}) is the optimal one (Cramer-Rao inequality), then always bigger than our 
two-dimensional Fisher information $I_2(\theta)$.

\section{Appendix}

\subsection{Integrability of $h$ in the proof of Theorem \ref{thm: sc}}
\begin{proof}
In this paragraph we prove that $\int \sup_{\theta'}|h(y,y')(\theta')|dQ^{Y,2}_\theta<+\infty$ where $h$ is given by (\ref{Tsc2}). At first
$$
\aligned
&Q^{Y,2}_\theta(y,y')\sup_{\theta'\in\Theta}\log \frac{Q^{Y,2}_\theta(y,y')}{Q^{Y,2}_{\theta'}(y,y')}\\
&\le \sup_{\beta,\beta'\in \bar O}\left(\sum_{i,j} \mu_\theta(i)p_{ij}(\theta)q_{\beta_i(\theta)}(y)q_{\beta_j(\theta)}(y')\right)
\log \frac{\sum_{i,j} \mu_\theta(i)p_{ij}(\theta)q_{\beta_i(\theta)}(y)q_{\beta_j(\theta)}(y')}
{q_{\beta}(y)q_{\beta'}(y')}\\
&\le \sup_{\beta,\beta'\in \bar O}\sum_{i,j} \mu_\theta(i)p_{ij}(\theta) q_{\beta_i(\theta)}(y)q_{\beta_j(\theta)}(y') \log \frac{q_{\beta_i(\theta)}(y)q_{\beta_j(\theta)}(y')}
{q_{\beta}(y)q_{\beta'}(y')}\\
&\le \sum_{i,j} \mu_\theta(i)p_{ij}(\theta) q_{\beta_i(\theta)}(y)q_{\beta_j(\theta)}(y')  \left(\sup_{\beta\in \bar O}\log \frac{q_{\beta_i(\theta)}(y)}
{q_{\beta}(y)} + \sup_{\beta'\in \bar O}\log \frac{q_{\beta_j(\theta)}(y')}
{q_{\beta'}(y')}\right)
\endaligned
$$
where the third line inequality follows by the convexity of $x\log x$. The last term is $\sigma(dy)\sigma(dy')$-integrable by {\bf (H4)}, i.e. $\sup_{\theta'\in\Theta}\log \frac{Q^{Y,2}_\theta(y,y')}{Q^{Y,2}_{\theta'}(y,y')}$ is bounded from above by a $Q^{Y,2}_\theta$-integrable function.

For the lower bound,
$$
\aligned
&\inf_{\theta'\in \Theta}\log \frac{Q^{Y,2}_\theta(y,y')}{Q^{Y,2}_{\theta'}(y,y')}\\
&\ge \inf_{\beta,\beta'\in \bar O}
\log \frac{\sum_{i,j} \mu_\theta(i)p_{ij}(\theta)q_{\beta_i(\theta)}(y)q_{\beta_j(\theta)}(y')}
{q_{\beta}(y)q_{\beta'}(y')}\\
&\ge \inf_{\beta,\beta'\in \bar O}\sum_{i,j} \mu_\theta(i)p_{ij}(\theta)  \log \frac{q_{\beta_i(\theta)}(y)q_{\beta_j(\theta)}(y')}
{q_{\beta}(y)q_{\beta'}(y')}\\
&\ge \sum_{i,j} \mu_\theta(i)p_{ij}(\theta)  \left(\inf_{\beta\in \bar O}\log \frac{q_{\beta_i(\theta)}(y)}
{q_{\beta}(y)} + \inf_{\beta'\in \bar O}\log \frac{q_{\beta_j(\theta)}(y')}
{q_{\beta'}(y')}\right)\\
&= \sum_i \mu_\theta(i) \inf_{\beta\in \bar O}\log \frac{q_{\beta_i(\theta)}(y)}{q_{\beta}(y)} + \sum_j \mu_\theta(j) \inf_{\beta\in \bar O}\log \frac{q_{\beta_j(\theta)}(y')}{q_{\beta'}(y')}
\endaligned
$$
and by {\bf (H4)}
$$\aligned
&\int_S |\inf_{\beta\in \bar O}\log \frac{q_{\beta_i(\theta)}(y)}{q_{\beta}(y)}|d q_{\beta_j(\theta)}\\
&\le
\int_S \sup_{\beta\in \bar O}|\log \frac{q_{\beta_j(\theta)}(y)}{q_{\beta}(y)}| dq_{\beta_j(\theta)} + \int_S |\log \frac{q_{\beta_j(\theta)}(y)}{q_{\beta_i(\theta)}(y)}| dq_{\beta_j(\theta)}<+\infty
\endaligned$$
Then $\inf_{\theta'}\log \frac{Q^{Y,2}_\theta(y,y')}{Q^{Y,2}_{\theta'}(y,y')}$ is bounded from below by a $Q_\theta^{Y,2}(dy,dy')$-integrable function. Combining the upper and lower controls above, $\|h(y,y')\|_{\Theta}$ is $Q_\theta^{Y,2}(dy,dy')$-integrable.
\end{proof}

\subsection{The proof of Theorem \ref{thm: clt for H}}
\begin{proof}
Let $Z_k^{(2)}=(Z_k, Z_{k+1})=((X_k, Y_{k}), (X_{k+1}, Y_{k+1}))$ which is also a  Markov chain, by CLT for Markov chains (cf. Meyn-Tweedie \cite[Theorem 17.0.1]{meyn2012}),

\begin{equation}\label{A71}
\sqrt{n} \left(L_n^{Y,2}(y,y')- Q_\theta^{Y,2}(y,y')\right)_{(y,y')\in S^2} \xrightarrow{\text{in law}} \xi=(\xi(y,y'))_{(y,y')\in S^2} 
\end{equation}
where $\xi$ is the centered Gaussian vector with the covariance matrix given in Theorem \ref{thm: clt for H}. We divide the proof into two steps.

{\bf Step 1.} By Taylor's formula of order 2, for any $x>0$, there exists $0\leq t\leq 1$, such that 
$$x\log x=(x-1)+{1\over 2}\cdot{1\over t+(1-t)x}\cdot(x-1)^2.$$ Since 
$$\sum_{(y, y')}\left(\frac{L_n^{Y, 2}(y, y')}{Q_{\theta}^{Y, 2}(y, y')}-1\right)Q_{\theta}^{Y, 2}(y, y')=0,$$ we have
\begin{eqnarray*}
	H(L_n^{Y, 2}|Q_{\theta}^{Y,2})&=&\sum_{(y, y')}\frac{L_n^{Y, 2}(y, y')}{Q_{\theta}^{Y, 2}(y, y')}\log\frac{L_n^{Y, 2}(y, y')}{Q_{\theta}^{Y, 2}(y, y')}\cdot Q_\theta^{Y, 2}(y, y')\\
	&=&\frac{1}{2}\sum_{(y, y')}\frac{1}{t_n(\omega)+(1-t_n(\omega))\frac{L_n^{Y, 2}(y, y')}{Q_{\theta}^{Y, 2}(y, y')}}\left(\frac{L_n^{Y, 2}(y, y')}{Q_{\theta}^{Y, 2}(y, y')} -1 \right)^2\cdot Q_{\theta}^{Y, 2}(y, y').\\
\end{eqnarray*}
By the law of large number,
$$\frac{1}{t_n(\omega)+(1-t_n(\omega))\frac{L_n^{Y, 2}(y, y')}{Q_{\theta}^{Y, 2}(y, y')}}\xrightarrow{a.s.}1.$$
Then for any $\varepsilon>0$,

\begin{equation}\label{A72}\pp_\theta\left((1-\varepsilon)\frac{1}{2}\chi^2(L_n^{Y, 2} | Q_\theta^{Y, 2})\le H(L_n^{Y, 2}|Q_{\theta}^{Y,2})\le (1+\varepsilon)\frac{1}{2}\chi^2(L_n^{Y, 2} | Q_\theta^{Y, 2})\right) \rightarrow 1.\end{equation}

{\bf Step 2.} By (\ref{A71}),
\begin{align*}
n\chi^2(L_n^{Y, 2} | Q_\theta^{Y, 2})&=\sum_{(y, y')}\frac{[\sqrt{n}(L_n^{Y, 2}(y, y')-Q_\theta^{Y, 2}(y, y'))]^2}{Q_\theta^{Y, 2}(y, y')}\\
&\xrightarrow{\text{in law}}\sum_{(y, y')}\frac{\xi^2(y, y')}{Q_\theta^{Y, 2}(y, y')}
\end{align*}
Thus by (\ref{A72}), we obtain 
$$
n H(L_n^{Y,2}|Q^{Y,2}_\theta)\xrightarrow{\text{in law}}\frac 12\sum_{(y, y')}\frac{\xi^2(y, y')}{Q_\theta^{Y, 2}(y, y')}.
$$
The proof of Theorem \ref{thm: clt for H} is completed. 
\end{proof}

\subsection{Proof of Proposition \ref{prop55}}
\begin{defn}
Dobrushin's ergodic coefficient of the transition probability matrix $P$ is defined by
$$r_D(P)=\sup_{i, j\in \mm}||P(i, \cdot)-P(j, \cdot)||_{tv}$$
\end{defn}


\begin{proof}[Proof of Proposition \ref{prop55}] (a) For any $g : S^2\to \rr$, 
\begin{align*}
&\ee(\sum_{(y, y')\in S^2}\xi(y, y')g(y, y'))^2\\
&=\var(g(Y_0, Y_1))+2\sum_{k=1}^{+\infty}\cov(g(Y_0, Y_1), g(Y_k, Y_{k+1}))\\
&\le \var(g)+2\sum_{k=1}^{+\infty} \rho_2(P_\theta^{k-1}) \var(g)\\
&\leq\var(g)+2\sum_{k=1}^{+\infty}\sqrt{r_D(P_\theta^{k-1})}\var(g).\\
\end{align*}
But by {\bf (H1)}, $r_D(P_\theta^k)\le (1-\kappa)^{[k/n_0]}$, then 
\begin{equation}\label{A74}
\ee(\sum_{(y, y')\in S^2}\xi(y, y')g(y, y'))^2\le \left(1+ \frac{2n_0}{1-\sqrt{1-\kappa}}\right)\var(g).
\end{equation}
Hence
\begin{align*}
\ee|\zeta|^2&=\sum_{y, y'}\frac{\ee\xi^2(y, y')}{Q_\theta^{Y, 2}(y, y')}\\ 
&\leq\sum_{(y, y')\in S^2}\left(1+ \frac{2n_0}{1-\sqrt{1-\kappa}}\right)\cdot\frac{\var(1_{(y, y')})}{Q_\theta^{Y, 2}(y, y')}\\
&=\sum_{(y, y')\in S^2}\left(1+ \frac{2n_0}{1-\sqrt{1-\kappa}}\right)\cdot\frac{Q_\theta^{Y, 2}(y, y')(1 - Q_\theta^{Y, 2}(y, y'))}{Q_\theta^{Y, 2}(y, y')}\\
&=\left(1+ \frac{2n_0}{1-\sqrt{1-\kappa}}\right)\cdot(|S^2|-1). 
\end{align*}

(b).
Let $\Gamma_\zeta$ be the covariance matrix of gaussian vector $\zeta$, for any $\lambda\in \rr^{S^2}$, we have
$\ee e^{\langle \zeta, \lambda\rangle}=e^{{1\over 2}\langle \lambda,  \Gamma_\zeta\lambda\rangle}$. Furthermore, let $(e_k)_{1\le k\le |S|^2}$  be the  orthonormal basis of eigenvectors of $\Gamma_\zeta$ associated to the eigenvalues $(\lambda_k)$. $\zeta_k:=\langle \zeta, e_k\rangle$ are i.i.d. of law $\mathcal{N}(0,1)$. So we have
$$\ee e^{a|\zeta|^2}=\prod_{k=1}^{|S^2|}\ee e^{a \lambda_k |\zeta_k|^2}=\prod_k\frac{1}{\sqrt{1-2\lambda_ka}}.$$
Notice that if $Z\sim\NN(0, 1)$, for any $\lambda<1/2$, $\ee e^{\lambda Z^2}=\frac{1}{\sqrt{1-2\lambda}}$. So we have
$$\ee e^{a|\zeta|^2}=\prod_k\frac{1}{\sqrt{1-2\lambda_ka}}$$
where it follows 
\begin{equation}\label{eq: loge}
\log\ee e^{a|\zeta|^2}=-\sum_{k=1}^{|S^2|}\frac{1}{2}\log(1-2\lambda_ka).
\end{equation}
Let $\lambda_{\max}$ be the maximum eigenvalue of $\Gamma_\zeta$. Since for any $z\in \mathbb{R}^{S^2}$ with $|z|=1$, 
\begin{align*}
\langle z, \Gamma_\zeta z\rangle 
&= \ee \left(\sum_{(y,y')\in S^2} z(y,y') \frac{\xi(y,y')}{\sqrt{Q^{Y,2_\theta}(y,y')}}\right)^2\\
&\leq 1+ \frac{2n_0}{1-\sqrt{1-\kappa}}
\end{align*}
by (\ref{A74}). We obtain thus $\lambda_{\max} \le 1+ \frac{2n_0}{1-\sqrt{1-\kappa}}$. 

 By (\ref{eq: loge}), we get for any $a\in(0, 1/(2\lambda_{\max}))$,  
\begin{align*}
\log \ee e^{a |\zeta|^2}&\leq (\sum_k\lambda_k) a +\sum_k{1\over2}(2\lambda_k a)^2\cdot\frac{1}{1-2\lambda_k a }\\
&\le \left( a+ a^2\cdot\frac{2\lambda_{\max}}{1-2\lambda_{\max} a}\right) \ee |\zeta|^2
\end{align*}

For $c>1$ fixed, we have for any $a\in(0, 1/(2\lambda_{\max}))$, 
\begin{align*}
\pp(|\zeta|^2>c\ee|\zeta|^2)&\leq e^{- a (c-1)\ee|\zeta|^2}\cdot \ee e^{ a(|\zeta|^2-\ee |\zeta|^2)}\\
&\leq \exp\left(- \left(a (c-1)- \frac{2\lambda_{\max}a^2}{1-2\lambda_{\max}a}\right)\ee |\zeta|^2\right)
\end{align*}
By optimization over $a\in (0, \frac{1}{2\lambda_{\max}})$,
$$\pp(|\zeta|^2>c\ee|\zeta|^2)
\le\exp\left(-\frac{(\sqrt{c}-1)^{2}}{2\lambda_{\max}}\cdot\ee |\zeta|^2\right).$$
\end{proof}

\bibliographystyle{unsrt}

\end{document}